\documentclass[reqno]{amsart}
\usepackage[english]{babel}
\usepackage{amscd,amssymb,amsmath,amsfonts,latexsym,amsthm}
\usepackage{inputenc}
\usepackage{graphicx}
\usepackage{hyperref,amsthm,amsmath,amssymb,color,multirow,graphicx}
\usepackage[fontsize=11pt]{scrextend}

\newtheorem{thm}{Theorem}[section]

\newtheorem{lem}[thm]{Lemma}

\theoremstyle{definition}
\newtheorem{defn}[thm]{Definition}

\numberwithin{equation}{section}

\newcommand{\R}{\mathbb R}

\newcommand{\be}{\begin{equation}}

\newcommand{\ee}{\end{equation}}

\begin{document}

\sloppy

\begin{flushright}
	\begin{tabular}{ll}
		\textsf{} &  \\
		\textsf{\ \pageref{firstpage}-\pageref{lastpage}}\\
		{}\\
\end{tabular}
\end{flushright}

\begin{center}
\textbf{\large  A GENERALIZED TIME-FRACTIONAL KURAMOTO-SIVASHINSKY EQUATION IN THE SCHWARTZ SPACE: GLOBAL SOLVABILITY AND STABILITY}\\

 \vspace{5mm}

{\small
\begin{tabular}{p{9cm}}
    {\bf Ashurov  R.R.}\\
    V.I. Romanovskiy Institute of Mathematics,\\
   Uzbekistan Academy of Sciences,\\
  Engineering School, Central Asian University,  Tashkent, 111221, Uzbekistan.\\
   ashurovr@gmail.com\\
    \\[5mm]
    {\bf Sobirov Z.A.} \\
   National University of Uzbekistan, Tashkent, Uzbekistan,\\
   V.I. Romanovskiy Institute of Mathematics,\\
   Uzbekistan Academy of Sciences.\\
   sobirovzar@gmail.com\\
   \\[5mm]
    {\bf Norkulova R.B.}\\
   V.I. Romanovskiy Institute of Mathematics, \\
   Uzbekistan Academy of Sciences.\\
   norkulovarushana1988@gmail.com\\
    
\end{tabular}
}
\end{center}
\vspace{5mm}

\textbf{Abstract.} 
We study the Cauchy problem for a generalized time-fractional Kuramoto--Sivashinsky equation on \(\mathbb R\) with the regularized Caputo derivative in the Schwartz space. The linear problem is solved by the Fourier transform and a Mittag--Leffler representation, with preservation of the Schwartz class. For the nonlinear problem, local solvability is established by successive approximations with convergence in the full Schwartz topology. A continuation principle accounting for the complete fractional memory is developed: the history term is estimated uniformly with respect to the continuation point, which yields global solvability on every finite time interval. An \(L^2(\mathbb R)\) stability estimate is also obtained, implying uniqueness.




\makeatletter
\renewcommand{\@evenhead}{\vbox{\thepage \hfil {\it  Ashurov R.R., Sobirov Z.A., Norkulova R.B.}  \hrule }}
\renewcommand{\@oddhead}{\vbox{\hfill
{\it  Ashurov R.R., Sobirov Z.A., Norkulova R.B. }\hfill
\thepage \hrule}} \makeatother

\label{firstpage}

\section{INTRODUCTION}
Nonlinear evolution equations play a fundamental role in the mathematical modeling of complex physical, chemical, and engineering processes (\cite{QASEM}, \cite{Benlahsen}, \cite{Kawahara}).
The Korteweg--de Vries, Kuramoto--Sivashinsky, and Kawahara equations are among the most extensively studied nonlinear evolution equations. These equations arise in the study of nonlinear wave propagation, instability mechanisms, turbulence, and pattern formation in fluid dynamics, plasma physics, and other dissipative media
(\cite{ElWakil2010}, \cite{Kruslov}, \cite{Clifford}).

Among such equations, the Kuramoto--Sivashinsky equation, originally introduced independently by Yoshiki Kuramoto and Gregory Sivashinsky, occupies an important place in the theory of nonlinear dissipative systems.
This equation arises in various applications, including plasma physics, flame front propagation, thin-film dynamics, reaction--diffusion systems, and fluid instability theory \cite{Sivashinsky}.
In its classical form, the Kuramoto--Sivashinsky equation contains a fourth-order dissipative term, a second-order destabilizing term, and a nonlinear convective term.
Due to its rich mathematical structure and important physical applications, this equation has been extensively studied from the viewpoints of existence, uniqueness, stability, and asymptotic behavior of solutions in various functional spaces.

In particular, R.~Adams \cite{Adams} investigated the well-posedness of a generalized Kuramoto--Sivashinsky equation and established the existence, uniqueness, and stability of solutions in appropriate functional spaces.
Furthermore, G.~M.~Coclite and L.~Di Ruvo \cite{Coclite} studied the existence of classical solutions and analytical properties of the Kuramoto--Sivashinsky equation in the presence of anisotropic effects.
Several numerical approaches have also been proposed for the efficient approximation of solutions.
For example, the extrapolated collocation algorithm proposed by Shallu and V.~K.~Kukreja \cite{Shallu} provides an efficient numerical technique for computing accurate approximations.

In recent years, fractional calculus has become an important tool for modeling memory and hereditary properties of complex media.
Fractional derivatives provide effective mathematical models for anomalous diffusion, viscoelasticity, and nonlocal transport processes.
In particular, the Caputo fractional derivative is widely used in applied problems because it allows physically meaningful initial conditions.
The introduction of fractional time derivatives into nonlinear evolution equations leads to models that incorporate memory effects in a natural way.
Time-fractional Kuramoto--Sivashinsky type equations are relevant to models arising in fluid mechanics, plasma physics, and materials science
(\cite{Shah}, \cite{Kuramoto}, \cite{SIVASHINSKY}).
The presence of a fractional time derivative substantially changes the analytical structure of the problem and affects the regularity, stability, and long-time behavior of solutions.

One of the natural functional settings for studying such problems on the whole real line is the Schwartz space.
This space consists of infinitely differentiable rapidly decreasing functions and has particularly convenient properties with respect to the Fourier transform.
In particular, the Fourier transform is an automorphism of the Schwartz space, which makes this space well suited for the analysis of linear and nonlinear partial differential equations with constant coefficients.

The main objective of the present work is to establish the solvability and stability of the Cauchy problem for a nonlinear generalized time-fractional Kuramoto--Sivashinsky equation with the regularized Caputo derivative:
\begin{equation}\label{1-tenglama}
\partial_{0t}^{\beta}u
+a^2u_{xxxx}
+bu_{xxx}
+cu_{xx}
+du_x
+ku
+\gamma uu_x
=f(x,t),
\qquad
x\in\mathbb{R},\quad 0<t\leq T,
\end{equation}
with the initial condition
\begin{equation}\label{1-tenhglama IC}
u(x,0)=\varphi(x),
\qquad x\in\mathbb{R},
\end{equation}
where $0<\beta<1$, $a>0$, $\gamma\neq0$, and $b,c,d,k,$ are real constants.

For the corresponding linear problem, we use the Fourier transform to derive an explicit representation of the solution in terms of Mittag--Leffler functions and prove that the Schwartz class is preserved.
For the nonlinear problem, a successive approximation scheme is constructed.
A priori estimates in the countable family of seminorms defining the Schwartz topology are obtained, and the approximations are shown to converge in the full Schwartz topology.
This yields the local solvability result stated in Theorem~\ref{nonlinear}.

A principal difficulty in passing from local to global solvability is caused by the nonlocal character of the regularized Caputo derivative.
In contrast to the classical case, restarting the problem at an intermediate time requires the complete memory accumulated from the initial time.
To handle this difficulty, we decompose the fractional derivative into a shifted regularized Caputo derivative and a history term.
An explicit representation of the history term is derived, together with an estimate that is uniform with respect to the continuation point.
Combining this estimate with spatial a priori bounds, we obtain a continuation step whose length can be chosen independently of the continuation point.
This makes it possible to continue the local solution finitely many times and proves global solvability on every finite time interval; see Theorem~\ref{Global}.
Finally, an $L^2(\mathbb{R})$ continuous-dependence estimate is established, which in particular implies uniqueness; see Theorem~\ref{Stability}.

Despite the importance of time-fractional Kuramoto--Sivashinsky equations,
rigorous analytical results concerning the corresponding Cauchy problems
remain comparatively limited. Tawhari \cite{QASEM} studied a time-fractional
Kuramoto--Sivashinsky equation by the residual power series transform method
and obtained approximate analytical solutions together with a convergence
analysis of the approximation procedure. Other available studies are mainly
devoted to numerical methods, approximation procedures, or special solutions.
Veeresha and Prakasha \cite{Veeresha} applied the q-homotopy analysis transform
method, while Choudhary et al. \cite{Chouldhary} developed a higher-order stable
numerical approximation based on quintic B-splines. Hosseininia et al.
\cite{Hossein} considered a  variable-order two-dimensional fractional
Kuramoto--Sivashinsky equation. Exact or special solutions were obtained by
Sahoo and Ray \cite{SAHOO}, Aychluh and Ayalew \cite{MINILIK}, and by the
invariant subspace method in \cite{OUHADAN}. Recent numerical developments also
include compact difference schemes on graded meshes \cite{Wang}. We also mention
\cite{Richard}, where a fractional Kuramoto--Sivashinsky type equation involving
the Atangana--Baleanu derivative was investigated.

The present work addresses a different analytical problem. We study the
Cauchy problem on the whole real line and establish solvability in the full
Schwartz topology. In particular, successive approximations are shown
to converge with respect to the seminorms defining $S(\mathbb R)$.

The principal difference from the classical generalized
Kuramoto--Sivashinsky problem is encountered in the global continuation
argument. For the equation of integer-order, a local solution can be restarted
at an intermediate time; see, for example, Adams \cite{Adams}, where
well-posedness is established by a patching procedure. In the time-fractional
case, such a direct restart is not possible because the regularized Caputo
derivative retains the memory accumulated from the initial time.
To overcome this difficulty, we decompose the fractional derivative into a
shifted regularized Caputo derivative and an explicit history term. The latter
is estimated uniformly with respect to the continuation point. This yields
a continuation step whose length is independent of the continuation point
and consequently gives global solvability on every finite time interval.
Together with the local convergence in the full Schwartz topology and the
$L^2(\mathbb R)$ stability estimate, this constitutes the main contribution
of the paper.

\section{PRELIMINARIES}
This section presents attendant lemmas and the necessary definitions used in this work.
\begin{defn}\label{1def3}(\cite{John},\cite{Elias},p.135) The Schwartz space $S({{\mathbb R}})$ is the topological vector space of functions $g:{{\mathbb R}}\to \mathbb R $ such that $g\in {{C}^{\infty }}({{\mathbb R}})$ and
\begin{equation}\label{S(R)}\sup_{x\in \mathbb R}{|x|^{k}}\left|\frac{d^n g(x)}{dx^n}\right|<\infty , \end{equation}	
for every  $k,n\in \mathbb N\cup \{0\}$.\end{defn}

The Schwartz space is a natural one to use for the Fourier transform. Differentiations and multiplication exchange roles under the Fourier transform due to the properties of smoothness and rapid decrease. As a result, the Fourier transform is an automorphism of the Schwartz space.

\begin{defn}(\cite{John},\cite{Elias})
The Fourier transform of a function $f\in L^1({\mathbb R})$ is a function $\hat{f}:{\mathbb R}\to \mathbb C$ defined by  
\begin{equation}\hat{f}(\lambda )\equiv \mathbf F[f](\lambda):=\frac{1}{\sqrt{2\pi}}\int\limits_{-\infty }^{+\infty }{f(x){{e}^{-i\lambda x}}dx}. \end{equation}		
The inverse Fourier transform is defined by 
\begin{equation}f(x)\equiv \mathbf F^{-1}[\hat f](x):=\frac{1}{\sqrt{2\pi} }\int\limits_{-\infty }^{+\infty }{\hat{f}(\lambda ){{e}^{i\lambda x}}d\lambda, }\end{equation} provided the right-hand side exists.

We  use $x$ to denote the spatial variable and $\lambda $  the variable in the Fourier transform.
\end{defn}

\begin{thm}(\cite{Elias},p.137)
 \label{thm1}
The Fourier transform and its inverse map the space $S$ onto itself in a one-to-one, linear, and continuous manner. \end{thm}

The Riemann--Liouville fractional integral of order $\beta>0$ of a function $g$ is defined by (see \cite{Kilbas2006}, p. 69)
$$I_{0t}^\beta g(t)=\frac{1}{\Gamma(\beta)}\int\limits_{0}^{t}\frac{g(\xi)}{(t-\xi)^{1-\beta}}d\xi, \quad t>0,$$
provided that the integral on the right-hand side exists.

If $f\in L^p(0,T)$, $1\leq p\leq +\infty$, then (see Lemma 2.3 in \cite{Kilbas2006})
$$(I_{0t}^{\alpha}I_{0t}^{\beta} f ) (t)=(I_{0t}^{\alpha+\beta}f)(t).$$

For $0<\beta < 1$, an operator defined by  
\begin{equation}\label{Caputo}\partial_{0 t}^{\beta}g(t)=\frac{d}{dt}I_{0t}^{1-\beta}(g(t)-g(0))=\frac{1}{\Gamma(1-\beta)}\frac{d}{dt}\int\limits_{0}^{t}{\frac{g(\xi)-g(0)}{{{(t-\xi )}^{\beta}}}d\xi },\ \  t>0, \end{equation}		
is called the regularized Caputo fractional derivative, also known as the Dzherbashyan--Caputo fractional derivative \cite{Dzherbashyan1968}, provided that the right-hand side is well defined. Here $\Gamma(\beta)=\int\limits_{0}^{+\infty}x^{\beta-1}e^{-x}dx$ 
is  Euler's Gamma function.

If $g$ is absolutely continuous on $[0,T]$, one has  (\cite{Kilbas2006})
\begin{equation}\label{Caputo2}\partial_{0 t}^{\beta}g(t)=I_{0t}^{1-\beta}\frac{d}{dt}g(t)=\frac{1}{\Gamma(1-\beta)}\int\limits_{0}^{t}{\frac{g^\prime(\xi)}{{{(t-\xi )}^{\beta}}}d\xi },\ \  t>0. \nonumber \end{equation}		

Moreover, if $\partial_{0 t}^{\beta}g(t)\in C[0,T]$ then $g(t)-g(0)=I_{0 t}^{\beta}\partial_{0 t}^{\beta}g(t)$.

Next, we introduce the two-parameter function of the Mittag-Leffler type, which plays a very important role in the fractional calculus.

\begin{defn}(\cite{I.Podlubniy},\cite{Rudolf}) For $0<\alpha<1$, a two-parameter function of the Mittag-Leffler type is defined by the series expansion
\begin{equation}\label{ML-def}
{{E}_{\alpha ,\gamma  }(z)}=\sum\limits_{k=0}^{\infty }{\frac{{{z}^{k}}}{\Gamma (\alpha k+\gamma )}}, \quad		 z \in \mathbb{C}.
\end{equation}		
\end{defn}
If $\gamma= 1$, then the Mittag-Leffler function is called the one-parameter or classical Mittag-Leffler
function and it is denoted by $E_\alpha(z) = E_{\alpha,1}(z)$.

For $m=0,1,2,\ldots$, we denote by
$$
E_{\alpha,\gamma}^{(m)}(z)
:=\frac{d^m}{dz^m}E_{\alpha,\gamma}(z)
$$
the $m$th derivative of the Mittag--Leffler function with respect to its argument. By termwise differentiation of \eqref{ML-def}, one has (see \cite{Biolek2025})
\begin{equation}\label{ML-derivatives}
E_{\alpha,\gamma}^{(m)}(z)
=\sum_{j=0}^{\infty}
\frac{(j+m)!}{j!}
\frac{z^j}{\Gamma\bigl(\alpha(j+m)+\gamma\bigr)}.
\end{equation}
See also \cite{Rudolf,Gorenflo2002} for properties of the Mittag--Leffler function and its derivatives.

\begin{thm} 
  \label{ML-derivative-estimate}(see \cite[Theorem 1.6, p. 35]{I.Podlubniy} and \cite[Proposition 4, p. 34391; Proposition 6, p. 34393]{Biolek2025}).
Let $0<\alpha<2$, $\gamma\in\mathbb R$, and
$m=0,1,2,\ldots$. Let $\mu$ satisfy
$$
\frac{\pi\alpha}{2}<\mu<\min\{\pi,\pi\alpha\}.
$$
Then there exists a constant $C>0$ such that
$$
\left|E_{\alpha,\gamma}^{(m)}(z)\right|
\leq
\frac{C}{1+|z|^{m+1}},
\qquad
\mu\leq|\arg z|\leq\pi .
$$
\end{thm}



Now we present several fundamental inequalities that are frequently employed in our analysis (\cite{Evans}, p.706).

a) Young's inequality. Let $1<p, q<\infty$, $\frac{1}{p}+\frac{1}{q}=1.$ Then
$$ab\leq\frac{a^p}{p}+\frac{b^q}{q},\quad(a,b>0).$$

b) Young's inequality with $\varepsilon.$
$$ab\leq \varepsilon a^p+C(\varepsilon)b^q,\quad(a,b>0,\varepsilon>0)$$
for $C(\varepsilon)=(\varepsilon p)^\frac{-q}{p}  q^{-1}.$

c) H\"older's inequality. Assume $U\subset \mathbb R^n$, $1\leq p,q\leq \infty,\frac{1}{p}+\frac{1}{q}=1.$ If $u\in L^p(U)$ and $v\in L^q(U)$, then
$$\int\limits_{U} |uv|dx\leq \|u\|_{L^p (U)}\|v\|_{L^q (U)}.$$

\begin{lem} \label{Lemm_yakupov}\cite{Yakupov1975}
If $u \in S(\mathbb R)$ and  $\int\limits_{-\infty}^{+\infty}\bigg(\frac{{\partial}^N u}{\partial x^N}\bigg)^2 dx\leq M<+\infty$, then the following inequality holds.
\begin{equation}\label{eqlem7}
    \int\limits_{-\infty}^{+\infty}x^{2m}\bigg(\frac{\partial^k u}{\partial x^k}\bigg)^2 dx\leq c_1(k,m,M)\bigg(\int\limits_{-\infty}^{+\infty}x^{2m+2}u^2 dx\bigg)^{\frac{m}{m+1}}+c_2(k,m,M)\bigg(\int\limits_{-\infty}^{+\infty}x^{2m+2}u^2dx\bigg)^\frac{1}{2^k +1},\end{equation}
for $2mk\leq N$, where $c_j (k,m,M),  j=1,2$  are some positive numbers, while $m$, $k$ and $N$ are natural numbers.
\end{lem}

It is easy to see that from \eqref{eqlem7} it follows that
\begin{equation}\label{Asosiy lemma}   \int\limits_{-\infty}^{+\infty}x^{2m}\bigg(\frac{\partial^k u }{\partial x^k}\bigg)^2 dx\leq M_\varepsilon+\varepsilon \int\limits_{-\infty}^{+\infty}x^{2m+2}u^2 dx, \end{equation}
where, for $\varepsilon>0$, the positive constant $M_\varepsilon$ depends on $m,k,M$ and $N$.

Now we introduce seminorms which are important in defining topological convergence in $S(\mathbb R)$.
According to \cite{Yakupov1975}, we introduce these seminorms as follows
\begin{equation}\label{yarim norma}
|||u|||_{m,s}^2=\int\limits_{-\infty}^{+\infty}\bigg|\frac{\partial^m}{\partial x^m}u(x)\bigg|^2 dx+\int\limits_{-\infty}^{+\infty}\left(1+x^2\right)^s| u(x)|^2dx,
\end{equation}
where $m$ and $s$ are nonnegative integers.

Now we justify the above seminorms. It is clear that if $v(x)\in S(\mathbb R)$, then the seminorm $|||v|||_{m,s}$ is bounded for any $s,m\in \mathbb N\cup\{0\}.$ 

On the other hand, if $|||v|||_{m,s}$ is bounded for any $s,m\in \mathbb N\cup\{0\}$, then using the Cauchy–Bunyakovsky–Schwarz inequality and inequality \eqref{Asosiy lemma}, we have
$$ \sup_{x\in \mathbb R} \left|x^s\frac{d^m v(x)}{d x^m}\right|^2= 
2\sup_{x\in \mathbb R} \int\limits_{-\infty}^x\left( s\xi^{2s-1}\left(\frac{d^m v(\xi)}{d \xi^m}\right)^2
+\xi^{2s}\frac{d^{m} v(\xi)}{d \xi^{m}}\frac{d^{m+1} v(\xi)}{d \xi^{m+1}}\right)d\xi$$

$$\leq 2s\int\limits_{-\infty}^{+\infty} (1+x^2)^{s}\left(\frac{d^m v(x)}{d x^m}\right)^2dx$$

$$+ \left(\int\limits_{-\infty}^{+\infty} (1+x^2)^{s}\left(\frac{d^m v(x)}{d x^m}\right)^2dx\int\limits_{-\infty}^{+\infty} (1+x^2)^{s}\left(\frac{d^{m+1} v(x)}{d x^{m+1}}\right)^2dx\right)\leq const.$$

Thus, we can conclude that the condition \eqref{S(R)} in the definition of the Schwartz space $S(\mathbb R)$ is equivalent to the boundedness of the seminorms $|||v|||_{m,s}$ for all $m,s\in \mathbb N\cup\{0\}.$ 

We introduce the space of continuous functions $\mathbf u(t)=u(\cdot,t)$ as $SC(T)=C([0,T], S(\mathbb R))$. Following \cite{Yakupov1975}, we introduce the topological space $SC_\beta(T)=\{ u\in SC(T): \partial_{0t}^\beta u\in SC(T)\}$. Convergence in this space is defined by countably many seminorms $\max_{0\leq t\leq T}|||u (\cdot,t)|||_{m,s}$, $\max_{0\leq t\leq T}|||(\partial_{0t}^\beta u) (\cdot,t)|||_{m,s},$ where $m, s =0,1,2,...\ .$

The following lemma is a direct consequence of Theorem 2.4 in \cite{Zacher2012}.
\begin{lem} \label{LemmAli}
Let $H$ be a real Hilbert space, $0<\beta<1$ and $T>0$. Suppose that 
$$v\in L^2([0,T];H), v-v_0=I_{0t}^\beta w \mbox{  for some  } w\in L^2([0,T];H) \mbox{ and } v_0\in H. $$
Then 
$$ (v(t),\partial_{0t}^\beta v(t))_H\geq\frac{1}{2}\partial_{0t}^\beta \|v(t)\|_H^2$$
for almost every $t\in [0,T].$
\end{lem}

The next statement follows from the comparison principle of Li and Liu \cite[Theorem 4.10]{LiLiu2018}, together with the explicit solution formula for the corresponding linear fractional equation \cite[Proposition 4.3]{LiLiu2018}.
\begin{lem}[Fractional Gronwall inequality]\label{GRONWALL}\cite{LiLiu2018}
Let $0<\beta<1$, $c\geq0$, and let $y\in C[0,T]$ and $q\in C[0,T]$ be nonnegative and $\partial_{0t}^\beta y\in C[0,T]$. Suppose that
\begin{equation}\label{weakGronwallAss}
\partial_{0t}^{\beta}y(t)\leq cy(t)+q(t)
\end{equation}
holds on $[0,T]$. Then
\begin{equation}\label{weakGronwallConclusion}
y(t)\leq y(0)E_\beta(ct^\beta)
+\int_0^t(t-s)^{\beta-1}E_{\beta,\beta}\!\left(c(t-s)^\beta\right)q(s)\,ds.
\end{equation}
Consequently, for $0\leq t\leq T$,
\begin{equation}\label{weakGronwallSimple}
y(t)\leq y(0)E_\beta(ct^\beta)
+\Gamma(\beta)E_{\beta,\beta}(ct^\beta)I_{0t}^{\beta}q(t).
\end{equation}
\end{lem}

\begin{lem}\label{yangi lemma1}
Let $T>0$, and let the sequence of functions $g_i(t)\in C[0,T]$, $i=0,1,2,...,n$, satisfy the inequality
$$g_i(t)\leq a+bI_{0t}^\beta g_{i-1}({t}), \quad i=1,2,...,$$
where $a\geq 0,b\geq 0, 0<\beta\leq 1,$ and $I_{0t}^\beta$ denotes the Riemann-Liouville fractional integral of order $\beta$. Then for all $n\in \mathbb{N}$, the following estimate holds
\begin{equation}\label{yangi lemma}
g_n (t)\leq a\sum_{i=0}^{n-1} \frac{b^i t^{i\beta}}{\Gamma(i\beta +1)}+b^n I_{0t}^{n\beta} g_0 (t).
\end{equation}
\end{lem} 
\begin{proof} Using the given inequality recursively, we obtain
  $$g_1 (t) \leq a+b I_{0t}^\beta g_0 (t),$$
  For the next term,
$$g_2 (t) \leq a+bI_{0t}^\beta g_1 (t) \leq a+bI_{0t}^\beta \left(a+b I_{0t}^\beta g_0 (t)\right)$$  
$$=a+ab\frac{t^\beta}{\Gamma(\beta+1)}+b^2 I_{0t}^{2\beta}g_0 (t).$$ Similarly,
$$g_3 (t)\leq a+bI_{0t}^\beta g_2 (t)\leq a+bI_{0t}^\beta\left(a+ab\frac{t^\beta}{\Gamma(\beta+1)}+b^2 I_{0t}^{2\beta}g_0 (t)\right)$$
$$=a+ab\frac{t^\beta}{\Gamma(\beta+1)}+ab^2\frac{t^{2\beta}}{\Gamma(2\beta+1)}+b^3 I_{0t}^{3\beta}g_0 (t),$$
Proceeding by mathematical induction, we obtain, in general, 
$$g_n (t) \leq a+ab\frac{t^\beta}{\Gamma(\beta+1)}+ab^2\frac{t^{2\beta}}{\Gamma(2\beta+1)}+ab^3\frac{t^{3\beta}}{\Gamma(3\beta+1)}+....$$
$$+ab^{n-1}\frac{t^{(n-1)\beta}}{\Gamma((n-1)\beta+1)}+b^n I_{0t}^ {n\beta} g_0 (t).$$
This can be written in compact form as
\begin{equation}
g_n (t) \leq a\sum_{i=0}^{n-1}b^i\frac{t^{i\beta}}{\Gamma(i\beta+1)}+b^n
I_{0t}^{n\beta}g_0(t).\end{equation}
The lemma is proved.
\end{proof}

\section{CAUCHY PROBLEM FOR THE LINEAR EQUATION} 
We consider the initial value (Cauchy) problem for the linear part of the Kuramoto--Sivashinsky equation 
\begin{equation}\label{chiziqli tenglama}\partial _{0t}^{\beta }u+a^2u_{xxxx}+bu_{xxx}+cu_{xx}+du_x+ku=g(x,t), \quad x\in {\mathbb R},\quad 0<   t\le T,\end{equation}		
with the initial condition	
\begin{equation}\label{3-tenglama}u(x,0)=\varphi(x), \quad x\in {\mathbb R}, \end{equation} 
where  $0<\beta <1$, $a>0$, $b, c, d,$ and $k$ are real numbers.

\begin{thm}\label{linear}
Let $\varphi\in S(\mathbb R)$ and $g\in SC(T)$. Then the Cauchy problem
\eqref{chiziqli tenglama}--\eqref{3-tenglama} has a unique solution in the class $SC_\beta(T)$.
\end{thm}

\begin{proof}
We prove that the Cauchy problem has a solution using Fourier transforms. Applying the Fourier transform with respect to $x$ (Definition \ref{1def3}), we get the following
\begin{equation}\label{FT1}\partial _{0t}^{\beta }\hat{u}(\lambda ,t)+P(\lambda)\hat{u}(\lambda ,t)=\hat{g}(\lambda ,t), \end{equation}
\begin{equation}\label{FT2} \hat{u}(\lambda,0)=\hat \varphi(\lambda ), \end{equation}	
where  $P(\lambda)=a^2\lambda^4 - ib\lambda^3 -c\lambda^2 +id\lambda+k$, and the functions $\hat{u}(\lambda ,t),\,\,\hat{g}(\lambda ,t),\,\,\hat \varphi(\lambda )$ are, respectively, the images of the functions $u(x,t),\,g(x,t),\,\varphi(x)$. 

The solution to the Cauchy problem \eqref{FT1}-\eqref{FT2} is as follows (\cite{Kilbas2006},p.141)
\begin{equation}\label{yechim}\hat{u}(\lambda ,t)=\hat{\varphi}(\lambda )\cdot {{E}_{\beta ,1}}\left(-P(\lambda){{t}^{\beta}}\right)+\int\limits_{0}^{t}{\hat{g}(\lambda ,\tau ){{(t-\tau )}^{\beta -1}}{{E}_{\beta ,\beta }}\left(-P(\lambda)(t-\tau )^\beta\right)d\tau }.\end{equation}

It is clear that the real part of the function $P(\lambda)$ is positive for large $|\lambda|$. 
Hence, for fixed $t>0$ and sufficiently large $|\lambda|$, the values $-P(\lambda)t$ belong to a fixed closed sector where the estimates for the Mittag--Leffler functions stated above are valid.




Since
$$
-P(\lambda)=-a^2\lambda^4+O(|\lambda|^3),
\qquad |\lambda|\to\infty,
$$
we have $|\arg(-P(\lambda))|\to\pi$ as $|\lambda|\to\infty$. Fix $\mu$ such that
$\frac{\beta\pi}{2}<\mu<\beta\pi.$ Then there exists $R>0$ such that
$\mu\leq |\arg(-P(\lambda))|\leq\pi,$  $|\lambda|\geq R$.

Since multiplication by a positive number does not change the argument,
the same sector condition holds for
$-P(\lambda)t^\beta$, $t>0$.
Therefore, the estimates of Theorem~\ref{ML-derivative-estimate}
are applicable.

According to Theorem \ref{ML-derivative-estimate}, for $t>0$ the derivatives of the Mittag--Leffler functions in \eqref{yechim} satisfy the corresponding sectorial estimates. For $t=0$ the value $E_{\beta,1}(0)$ is constant. Since $P(\lambda)$ is a polynomial and all its derivatives have at most polynomial growth, the chain rule shows that, for every nonnegative integer $n$, all derivatives with respect to $\lambda$ of $ E_{\beta,1}\left(-P(\lambda)t^\beta\right)$ have at most polynomial growth with respect to $\lambda$, uniformly for $0\leq t\leq T$. Similarly, all derivatives with respect to $\lambda$ of $E_{\beta,\beta}\left(-P(\lambda)r^\beta\right)$ have at most polynomial growth with respect to $\lambda$, uniformly for $0\leq r\leq T$. 

Thus, these functions are multipliers of the Schwartz space $S(\R)$.

Since $r^{\beta-1}\in L^1(0,T)$ and $\hat g\in SC(T)$, the integral in \eqref{yechim} converges in every seminorm of $S(\R)$. So,  
$\hat u\in SC(T).$

Furthermore, directly from \eqref{FT1} it follows that $\partial_{0t}^\beta\hat u\in SC[0,T]$. Thus, $u\in SC_\beta(T)$.

Now, according to the properties of the direct and inverse Fourier transforms \cite{Elias}, we conclude that the Cauchy problem \eqref{chiziqli tenglama}-\eqref{3-tenglama} has a unique solution $u(x,t)\in SC_\beta (T)$. 

Uniqueness follows from the uniqueness of the scalar Cauchy problem \eqref{FT1}--\eqref{FT2} for each fixed $\lambda\in\R$.
\end{proof}

 \section{LOCAL SOLVABILITY OF THE CAUCHY PROBLEM FOR THE NONLINEAR EQUATION} 

 In this section, we consider the nonlinear time-fractional Kuramoto--Sivashinsky equation 
\begin{equation}\label{1-tenglamasolvability}\partial _{0t}^{\beta }u+Lu+\gamma uu_x=f(x,t), \quad x\in \mathbb R,\quad 0<   t\le T,\end{equation}
where $0<\beta <1$,  $Lu:=a^2u_{xxxx}+bu_{xxx}+cu_{xx}+du_x+ku,$
with the initial condition 
\begin{equation}u(x,0)=\varphi(x), \quad x\in \mathbb R,\quad 0\leq t\le T. \end{equation}

\begin{thm} \label{nonlinear} (Local solvability). 
Let ${\varphi (x)}\in S(\mathbb R)$ and ${f(t,x)}\in SC(T)$. Then 
the Cauchy problem is solvable in the space $SC_\beta(t_1)$, where $t_1$ depends on $ \|\varphi\|_2,\max_{0\leq t\leq T} \|f(\cdot,t)\|_2$ and the coefficients of the equation, where  
$$\|v\|_2^2 = \int\limits_{-\infty}^{+\infty}\left(v^2 (x)+\left(\frac{\partial^2 v(x)}{\partial x^2}\right)^2\right)dx.$$
\end{thm}

The proof of the theorem is based on a priori estimates. 

We consider the successive equation for $i\geq 1$ 
\begin{equation}\label{regulyarizatsiya}\partial _{0t}^{\beta }{{u}_{i}} +Lu_i=f(x,t)-\gamma{{u}_{i-1}}{{u}_{i-1,x}},\ \ x\in\mathbb R,\quad 0<t\leq T,\end{equation}			
\begin{equation}{{u}_{i}}(x,0)=\varphi(x), \quad x\in\mathbb R,\end{equation}
 $u_0(x,t)=\varphi(x).$
The existence of a solution follows from Theorem \ref{linear} for each $i$. Now we show that the sequence converges in the topology defined by the seminorms \eqref{yarim norma}.

Thus, we need to obtain countably many a priori estimates in the seminorms \eqref{yarim norma}. Based on these estimates, we show that $\{u_i\}$ is a Cauchy sequence. 

We put 
$$\|v\|^2=\int\limits_{-\infty}^{+\infty}v^2(x)dx.$$

 \subsection{Estimate for $\|u_i\|^2+\|u_{ixx}\|^2$.}
  
\begin{lem}\label{norma2}
    There exists a number $t_1>0$ depending on $ \|\varphi\|_2, \max_{0\leq t\leq T}\|f(\cdot,t)\|_2$, such that for $i>2$
    $$\|u_i\|^2+\|u_{ixx}\|^2\leq const <+\infty,$$
    for all $0\leq t\leq t_1$, where the constant does not depend on $t$ and $i$.
\end{lem}
\begin{proof}

Multiply equation \eqref{regulyarizatsiya} by $2{{u}_{i}}$ and integrate with respect to $x$ over ${\mathbb R}$
  \begin{equation}\label{tenglama2}2\int\limits_{-\infty }^{+\infty }{{{u}_{i}}\partial _{0t}^{\beta }{{u}_{i}}dx}
 +2\int\limits_{-\infty }^{+\infty }{{u}_{i}}{{{Lu}_{i}}dx}=-2\gamma\int\limits_{-\infty }^{+\infty }{{{u}_{i}}\frac{\partial u_{i-1} }{\partial x}{{u}_{i-1}}dx}+2\int\limits_{-\infty }^{+\infty }{f(x,t){{u}_{i}}dx}.\end{equation}	

Now we estimate the terms on the left-hand side

$$2\int\limits_{-\infty }^{+\infty }{{{u}_{i}}\partial _{0t}^{\beta }{{u}_{i}}dx\ge \partial _{0t}^{\beta }{{\left\| {{u}_{i}} \right\|}^{2}}},$$
$$\int\limits_{-\infty }^{+\infty }{{{u}_{i}}{{u}_{ixxx}}dx=0},\ \ \int\limits_{-\infty }^{+\infty }{{{u}_{i}}{{u}_{ix}}dx=0},$$
$$\int\limits_{-\infty }^{+\infty }{{{u}_{i}}{{u}_{ixxxx}}dx= {{\left\| {{u}_{ixx}} \right\|}^{2}}}, \ \ 
-2\int\limits_{-\infty }^{+\infty }{{u}_{i}}{{u}_{ixx}}dx\leq \|u_i\|^2+\|u_{ixx}\|^2.$$

We now examine the first term on the right-hand side of the equation  
\begin{equation}\label{QR1}
Q_0=\left|-2\gamma\int\limits_{-\infty }^{+\infty }{{{u}_{i}}\frac{\partial u_{i-1} }{\partial x}{{u}_{i-1}}dx}\right|\le2|\gamma |\sup_{x\in \mathbb R} \left| {{u}_{i-1}} \right|\cdot \int\limits_{-\infty }^{+\infty }{\left| {{u}_{i}} \right|}\left| \frac{\partial {{u}_{i-1}} }{\partial x} \right|dx.\end{equation}		

Now we estimate	$\sup_{x\in \mathbb R}\left|u_{i-1}(x,t)\right|.$ It is known that

\begin{equation} \label{1}
u_{i-1}^{2}(x,t)=2\int\limits_{-\infty }^{x}{\frac{\partial {{u}_{i-1}}(\xi ,t)}{\partial \xi}\cdot {{u}_{i-1}}}(\xi ,t )d\xi.\end{equation}

Applying the Cauchy--Schwarz inequality, we obtain
 
\begin{equation}\label{QR2}
\underset{x\in {{\mathbb R}}}{\mathop{\sup }}\,\left| {{u}_{i-1}}(x,t) \right|\le {{\left( 2\int\limits_{-\infty }^{+\infty }{\left| \frac{\partial {{u}_{i-1}}}{\partial x} \right|\cdot \left| {{u}_{i-1}} \right|dx} \right)}^{\frac{1}{2}}}\le {{\left( 4\int\limits_{-\infty }^{+\infty }{{{\left( \frac{\partial {{u}_{i-1}}}{\partial x} \right)}^{2}}dx\cdot \int\limits_{-\infty }^{+\infty }{u_{i-1}^{2}dx}} \right)}^{\frac{1}{4}}}.
\end{equation}
Using \eqref{QR1} and \eqref{QR2}, we arrive at  
$$Q_0\leq 2|\gamma| \sqrt 2\left(\int\limits_{-\infty}^{+\infty}u_i^2dx\right)^{\frac{1}{2}}\cdot\left(\int\limits_{-\infty}^{+\infty}\left(\frac{\partial u_{i-1}}{\partial x}\right)^2 dx\right)^{\frac{3}{4}}\cdot \left(\int\limits_{-\infty}^{+\infty}u_{i-1}^2 dx\right)^{\frac{1}{4}}.$$
For any function $\vartheta (x)\in S(\mathbb R)$, the following relation holds:

\begin{equation}\int\limits_{-\infty }^{+\infty }{{{\left( {\vartheta }'(x) \right)}^{2}}dx=-\int\limits_{-\infty }^{+\infty }{\vartheta (x){\vartheta }''(x)dx\le {{\left( \int\limits_{-\infty }^{+\infty }{{{\vartheta }^{2}}(x)dx\cdot {{\int\limits_{-\infty }^{+\infty }{\left( {\vartheta }''(x) \right)^2}}}dx} \right)}^{\frac{1}{2}}}}}. \end{equation}	

Therefore, using the Cauchy and H\"older inequalities on the right-hand side of \eqref{QR1}, we can obtain the following estimate.
$${{Q}_{0}}\leq 2|\gamma|\sqrt{2}{{\left( \int\limits_{-\infty }^{+\infty }{u_{i}^{2}dx} \right)}^{\frac{1}{2}}}\cdot {{\left( \int\limits_{-\infty }^{+\infty }{{{\left( \frac{{{\partial }^{2}}{{u}_{i-1}}}{\partial {{x}^{2}}} \right)}^{2}}dx} \right)}^{\frac{3}{8}}}\cdot {{\left( \int\limits_{-\infty }^{+\infty }{u_{i-1}^{2}dx} \right)}^{\frac{5}{8}}}$$
 $$\le \int\limits_{-\infty }^{+\infty }{u_{i}^{2}dx+}{{c}_{0}}\left( {{\left( \int\limits_{-\infty }^{+\infty }{{{\left( \frac{{{\partial }^{2}}{{u}_{i-1}}}{\partial {{x}^{2}}} \right)}^{2}}dx} \right)}^{2}}+{{\left( \int\limits_{-\infty }^{+\infty }{u_{i-1}^{2}dx} \right)}^{2}} \right).$$ 
Using the above inequalities, we get 
\begin{equation}
    \label{QR3}
\partial _{0t}^{\beta }{{\left\| {{u}_{i}} \right\|}^{2}}+2a^2 {{\left\| {{u}_{ixx}} \right\|}^{2}}
\leq c_1 \left( \|u_{i}\|^{2} +\|u_{ixx}\|^{2}\right)
+{{c}_{0}}\left(  \|{u_{i-1}\|^{2}}+\| {u}_{i-1,xx}\|^2  \right)^2+ \|{{f}\|^{2}},
\end{equation}
where the positive constants $c_1$ and $c_0$ depend only on the coefficients of the equation \eqref{1-tenglama}.
 
Now, for equation \eqref{regulyarizatsiya}, we take the second-order derivative with respect to $x$ of each term, then multiply each term by $2\frac{{{\partial }^{2}}{{u}_{i}}}{\partial {{x}^{2}}}$ and integrate over ${\mathbb{R}}$ with respect to $x$.

$$2\int\limits_{-\infty }^{+\infty }{\frac{{{\partial }^{2}}}{\partial {{x}^{2}}}\partial _{0t}^{\beta }{{u}_{i}}\cdot \frac{{{\partial }^{2}}{{u}_{i}}}{\partial {{x}^{2}}}dx
+2\int\limits_{-\infty }^{+\infty }{\frac{{{\partial }^{2}}}{\partial {{x}^{2}}}}}Lu_i\cdot \frac{{{\partial }^{2}}{{u}_{i}}}{\partial {{x}^{2}}}dx$$

\begin{equation}    \label{QR4}
=-2\gamma\int\limits_{-\infty }^{+\infty }{\frac{{{\partial }^{2}}}{\partial {{x}^{2}}}\left({{u}_{i-1}}{{u}_{i-1,x}}\right)\frac{{{\partial }^{2}}u_i}{\partial {{x}^{2}}}dx}+2\int\limits_{-\infty }^{+\infty }{\frac{{{\partial }^{2}}f}{\partial {{x}^{2}}}\cdot \frac{{{\partial }^{2}}u_i}{\partial {{x}^{2}}}dx}.\end{equation}

We now estimate the terms on the left side of \eqref{QR4}.
    $$2\int\limits_{-\infty }^{+\infty }{\frac{{{\partial }^{2}}}{\partial {{x}^{2}}}\partial _{0t}^{\beta }{{u}_{i}}\frac{{{ \partial }^{2}}u_i}{\partial {{x}^{2}}}dx\ge \partial _{0t}^{\beta }{{\int\limits_{-\infty }^{+\infty }{\left( \frac{{{\partial }^{2}}u_i}{\partial {{x}^{2}}} \right)}}^{2}}dx\ge \partial _{0t}^{\beta }{{\left\| {{u}_{ixx}} \right\|}^{2}}},$$

$$\int\limits_{-\infty }^{+\infty }{\frac{{{\partial }^{5}}u_i}{\partial {{x}^{5}}}\frac{{{\partial }^{2}}u_i}{\partial {{x}^{2}}}dx=0}, \int\limits_{-\infty }^{+\infty }{\frac{{{\partial }^{3}}u_i}{\partial {{x}^{3}}}\frac{{{\partial }^{2}}u_i}{\partial {{x}^{2}}}dx=0}, $$
$$ \int\limits_{-\infty }^{+\infty }{\frac{{{\partial }^{6}}u_i}{\partial {{x}^{6}}}\frac{{{\partial }^{2}}u_i}{\partial {{x}^{2}}}dx=  \int\limits_{-\infty }^{+\infty }{u_{ixxxx}^{2}dx= {{\left\| {{u}_{ixxxx}} \right\|}^{2}},}}$$

$$ -2\int\limits_{-\infty }^{+\infty }\frac{{{\partial }^{4}}u_i}{\partial {{x}^{4}}}\frac{{{\partial }^{2}}u_i}{\partial {{x}^{2}}}dx=  2\int\limits_{-\infty }^{+\infty }u_{ixxx}^{2}dx
\leq \sigma_1\|u_{ixxxx}\|^2+\frac{1}{\sigma_1}\|u_{ixx}\|^2,$$
where the constant $\sigma_1>0$ will be chosen later.

So, we have 
\begin{equation}
    \label{QR5}
\partial _{0t}^{\beta }{{\left\| {{u}_{ixx}} \right\|}^{2}} +(2a^2-|c|\sigma_1){\left\| {{u}_{ixxxx}} \right\|}^{2}
\leq  c_3 \|u_{ixx}\|^{2}+ 2|\gamma|\int\limits_{-\infty}^{+\infty}\left(u_{i-1}u_{i-1,x}\right)\frac{\partial^4 u_i}{\partial x^4}dx
+2\int\limits_{-\infty }^{+\infty }\frac{{{\partial }^{2}}}{\partial x^2} f(x,t)\frac{{{\partial }^{2}}u_i}{\partial {{x}^{2}}}dx.\end{equation}

We estimate the first term on the right side of \eqref{QR5} in a similar way to \eqref{QR3}.
$$\int\limits_{-\infty}^{+\infty}u_{i-1}u_{i-1,x}\frac{\partial^4 u_i}{\partial x^4}dx\leq \sigma_2\int\limits_{-\infty}^{+\infty}\left(\frac{\partial^4 u_i}{\partial x^4}\right)^2 dx+
C_3 (\sigma_2)\left( \int\limits_{-\infty}^{+\infty}\left(\frac{\partial^2 u_{i-1}}{\partial x^2}\right)^2 dx+\int\limits_{-\infty}^{+\infty} u_{i-1}^2 dx\right)^2.$$

Now, choosing $\sigma_1$ and $\sigma_2$ such that $2a^2-|c|\sigma_1-2|\gamma|\sigma_2>0$, we get

\begin{equation}\label{ikkinchi baho}\partial _{0t}^{\beta }\left\| {u}_{ixx} \right\|^{2} 
\leq c_4 \left( \|u_{i}\|^{2} +\|u_{ixx}\|^{2}\right) +c_5 \left( \|u_{i-1}\|^{2} +\|u_{i-1,xx}\|^{2}\right)^2
+\|f_{xx}\|^2.\end{equation}

From \eqref{QR3} and \eqref{ikkinchi baho}, we obtain
$$\partial_{0t}^\beta\left(\|{u_i\|^2+\|u_{ixx}\|^2}\right) 
\leq c_6\left(\|{u_i\|^2+\|u_{ixx}\|^2}\right)+c_7\left(\|{u_{i-1}\|^2+\|u_{i-1,xx}\|^2}\right)^2
+ c_8\left(\|f\|^2 +\|f_{xx}\|^2\right),$$
where the positive constants $c_6, c_7$ and $c_8$  depend only on the coefficients of the equation. 


 We introduce the following notation:

 $$U_0=\|\varphi\|^2_2,\quad  F(t) = \|f\|^2_2, \quad {{U}_{i}}(t)=\|u_i\|^2_2,  \quad i\geq 1.$$

We have 
\begin{equation}\label{tenglama5}\partial _{0t}^{\beta }U_i\le c_6U_i+ c_7{U_{i-1}^{2}}+c_8F.\end{equation}

Using the fractional Gronwall inequality from Lemma \ref{GRONWALL}, we get

$$U_i(t) \leq U_0 E_\beta (c_6 t^\beta)+\Gamma(\beta)E_{\beta,\beta}(c_6 t^\beta)I_{0t} ^\beta \left({c_7}{U_{i-1}^{2}}+c_8F\right)$$

Then, for $0<t_1<T$, $i\geq 1$ we have 
\begin{equation}\max_{0\leq t\leq t_1} {U_i(t)}
\leq \frac{c_7t_1^\beta}{\beta}E_{\beta,\beta}(c_6 T^\beta)\max_{0\leq t\leq t_1} U_{i-1}^2
+\frac{c_8T^\beta}{\beta}E_{\beta,\beta}(c_6 T^\beta)\max_{0\leq t\leq T}F(t)
+ U_0E_\beta (c_6 T^\beta).\end{equation}

We put $$A= \frac{c_8T^\beta}{\beta}E_{\beta,\beta}(c_6 T^\beta)\max_{0\leq t\leq T}F(t)
+ U_0E_\beta (c_6 T^\beta).$$

It is clear that $U_0\leq 2A.$ Let $\max_{0\leq t\leq t_1} U_{i-1}\leq 2A$ for some $i\geq 1$. Then
$$\max_{0\leq t\leq t_1} U_i\leq \frac{4c_7t_1^\beta}{\beta}E_{\beta,\beta}(c_6 T^\beta)A^2+A.$$

Choosing 
$$t_1\leq \left(\frac{\beta}{4Ac_7 E_{\beta,\beta}(c_6T^\beta)}\right)^{1/\beta},$$
we get $\max_{0\leq t\leq t_1} U_i\leq 2A.$ 

Thus, the inequality ${{U}_{i}}(t)=\|u_i\|^2+\|u_{ixx}\|^2\leq 2A$ holds for all $i$ and $0<t<t_1$.

For further considerations, we assume $0\leq t\leq t_1\leq T$. 
\end{proof}

\subsection{Estimates for higher derivatives.}
\begin{lem} \label{high derivatives}
    Let $t_1$ be the same as in Lemma \ref{norma2}. Then 
    $$\left\| \frac{\partial^n}{\partial x^n}u_i(\cdot,t)\right\|^2\leq M(n),$$
    for all $0\leq t\leq t_1$, where $n\geq 3$ and $M(n)$ is a positive constant.
\end{lem}
\begin{proof}
    
 Let us now carry out an induction on the order of the derivative. For $n\geq 1$ we put 
$$|||v|||_n^2=\int\limits_{-\infty}^{+\infty} \left( \frac{d^n v(x)}{d x^n}\right)^2dx,$$
and $|||v|||_0=\|v\|$.

We assume $|||u_i|||_j\leq M(j)<+\infty$ for all $0\leq j<n$, $i\in \mathbb N$, where $n\geq 3$.  Now we can take the $n$th derivative of each term of the equation with respect to $x$, then multiply by $2\frac{\partial^n u_i}{\partial x^n}$, and integrate with respect to $x$ in $\mathbb{R}$.

$$2\int\limits_{-\infty}^{+\infty}\partial_{0t}^\beta \frac{\partial^n u_i}{\partial x^n}\frac{\partial^n u_i}{\partial x^n}dx+2\int\limits_{-\infty}^{+\infty}\frac{\partial^{n} }{\partial x^{n}}Lu_i {\frac{\partial^n u_i}{\partial x^n }}dx$$

\begin{equation}\label{LEYBNITS}
=-2\gamma{\int\limits_{-\infty}^{+\infty} {\frac{\partial^n}{\partial x^n}}
{\left(u_{i-1}\frac{\partial {u_{i-1}}}{\partial x}\right)}\frac{\partial^n u_i}{\partial x^n}dx 
+2\int\limits_{-\infty}^{+\infty}\frac{\partial^n f}{\partial x^n}\frac{\partial^n u_i}{\partial x^n}dx}.\end{equation}

Now we estimate the terms on the left-hand side-hand side of this equality.

\begin{equation}
2\int\limits_{-\infty}^{+\infty}\frac{\partial^n u_i}{\partial x^n}\partial_{0t}^\beta \frac{\partial^n u_i}{\partial x^n}dx\geq\partial_{0t}^\beta
\int\limits_{-\infty}^{+\infty}\left(\frac{\partial^n u_i}{\partial x^n}\right)^2 dx=\partial_{0t}^\beta|||u_i|||_n^2,\end{equation}

\begin{equation}
2a^2\int\limits_{-\infty}^{+\infty}\frac{\partial^{n+4} u_i}{\partial x^{n+4}}\frac{\partial^n u_i}{\partial x^n}dx=2a^2\int\limits_{-\infty}^{+\infty}\left(\frac{\partial^{n+2} u_i}{\partial x^{n+2}}\right)^2 dx=2a^2|||u_i|||^2_{n+2},\end{equation}

\begin{equation}
\int\limits_{-\infty}^{+\infty}\frac{\partial^{n+3} u_i}{\partial x^{n+3}}\frac{\partial^n u_i}{\partial x^n}dx=0, \ \ \int\limits_{-\infty}^{+\infty}\frac{\partial^{n+1}u_i}{\partial x^{n+1}}\frac{\partial^n u_i}{\partial x^n}dx=0,\end{equation}

\begin{equation}
\left|\int\limits_{-\infty}^{+\infty}\frac{\partial^{n+2}u_i}{\partial x^{n+2}}\frac{\partial^n u_i}{\partial x^n}dx\right|=|||u_i|||_{n+1}^2\leq \sigma |||u_i|||_{n+2}^2+ \frac{1}{4\sigma}|||u_i|||_n^2.\end{equation}

Now, let us estimate the integral arising from the nonlinear term.
\begin{equation}\label{4.24}
2\int\limits_{-\infty}^{+\infty}\frac{\partial^n}{\partial x^n} \left(u_{i-1} \frac{\partial u_{i-1}}{\partial x}\right)\frac{\partial^n u_i}{\partial x^n}dx
=\int\limits_{-\infty}^{+\infty}\frac{\partial^n u_i}{\partial x^n}\frac{\partial^{n+1} }{\partial x^{n+1}}(u^2_{i-1})dx
=\int\limits_{-\infty}^{+\infty}\frac{\partial^{n+2}u_i}{\partial x^{n+2}} \frac{\partial^{n-1}} {\partial x ^{n-1}}{(u^2_{i-1})}dx.\end{equation}

If we apply Leibniz's formula to the last expression, we get the following
\begin{equation}
\frac{\partial^{n-1}}{\partial x^{n-1}}(u^2_{i-1})=2\frac{\partial^{n-2}}{\partial x^{n-2}}\left(u_{i-1}u_{i-1,x}\right)=2\sum\limits_{k=0}^{n-2}C^k _{n-2}\frac{\partial^k u_{i-1} }{\partial x^k}\frac{\partial^{n-k-1}}{\partial x^{n-k-1}}u_{i-1}.\end{equation}


Let us define

\begin{equation}
I_{k,j}=\int\limits_{-\infty}^{+\infty} \frac{\partial^{n+2}u_i}{\partial x^{n+2}}\frac{\partial^k u_{i-1}}{\partial x^k}\frac{\partial^j u_{i-1}}{\partial x^j} dx,
\end{equation}

 where $k+j=n-1$. Without loss of generality, we assume $j\leq k$. So,  $j\leq\frac{n-1}{2}\leq n-2.$ Then 
 
 \begin{equation}
|I_{kj}|\leq |\gamma|\sup_{x\in\mathbb{R}}\left|\frac{\partial^j u_{i-1}}{\partial x^j}\right| \int\limits_{-\infty}^{+\infty}\left|\frac{\partial^{n+2}u_i}{\partial x^{n+2}} \right|\left|\frac{\partial^k u_{i-1}}{\partial x^k}\right|dx.  \end{equation}

Now we estimate the supremum. Taking into account $j\leq n-2$, and so $|||u_i|||_{j} \leq M(j), |||u_i|||_{j+1} < M(j+1)$, we have

\begin{equation} 
\sup_{x\in \mathbb R}\left|\frac{\partial^j u_{i-1}}{\partial x^j}\right|\leq\sqrt{2}\left(\int\limits_{-\infty}^{+\infty}\left(\frac{\partial^j u_{i-1}}{\partial x^j}\right)^2dx\int\limits_{-\infty}^{+\infty}\left(\frac{\partial^{j+1}u_{i-1}}{\partial x^{j+1}}\right)^2 dx\right)^{\frac{1}{4}}\leq B,\end{equation}
where $B$ is a constant independent of $t$ and $i$.
Thus,
$\left|I_{k,j}\right|\leq B |||u_i|||_{n+2} |||u_{i-1}|||_{k},$
and applying Young's inequality,

\begin{equation}|I_{k,j}|\leq\frac{ B}{4\sigma}\int\limits_{-\infty}^{+\infty}\left(\frac{\partial^{k} u_{i-1}}{\partial x^{k}}\right)^2 dx+\sigma \int\limits_{-\infty}^{+\infty} \left(\frac{\partial^{n+2}u_i}{\partial x^{n+2}}\right)^2 dx\leq C(\sigma)+\sigma|||u_i|||_{n+2}^2,\end{equation}
where $\sigma>0$ will be chosen later.

The last integral on the right-hand side of \eqref{LEYBNITS} is estimated using the Cauchy inequality:

\begin{equation}
\left|2\int\limits_{-\infty}^{+\infty}\frac{\partial^n f}{\partial x^n}\frac{\partial^n u_i}{\partial x^n}dx\right|
\leq \int\limits_{-\infty}^{+\infty}\left(\frac{\partial^n f}{\partial x^n}\right)^2 dx+\int\limits_{-\infty}^{+\infty} \left(\frac{\partial^n u_i}{\partial x^n}\right)^2 dx.\end{equation}

Summarizing the above estimates and choosing $\sigma$ appropriately, we have
\begin{equation}
\partial_{0t}^\beta |||u_i|||_n^2 \leq B|||u_{i}|||_n^2+|||f|||_n^2 +B_2.\end{equation}

Using Gronwall's inequality (see Lemma \ref{GRONWALL}), we obtain the desired estimate $ |||u_i|||_n^2\leq\widetilde{C},  t\in [0,t_1].$

\end{proof}

\subsection{Estimate for $\int\limits_{-\infty}^{+\infty} x^{2m} u^2_i dx$.}

\begin{lem} \label{estimate x_mu}
    Let $t_1$ be the same as in Lemma \ref{norma2}. Then 
    $$\int\limits_{-\infty}^{+\infty} x^{2m} u^2_i dx\leq M_1(m),$$
    for all $0\leq t\leq t_1$, where $m\geq 1$ and $M_1(m)$ is a positive constant.
\end{lem}
\begin{proof}

We do this in a similar way to the previous estimates. Without loss of generality, we suppose $m\geq 2$. So, we multiply each term of the equation by $2 x^{2m} u_i $ and integrate with respect to $x$ in $\mathbb R$.
\begin{equation}\label{ESTx^2m}
2\int\limits_{-\infty}^{+\infty} x^{2m} u_i \partial^\beta _{0t} u_i dx+2\int\limits_{-\infty}^{+\infty} x^{2m} u_i Lu_{i}dx
=-2\gamma\int\limits_{-\infty}^{+\infty} x^{2m} u_i u_{i-1} u_{i-1,x} dx+2\int\limits_{-\infty}^{+\infty} x^{2m} u_i f(x,t)dx.
\end{equation}

First, we estimate the first integral on the right-hand side of the equality \eqref{ESTx^2m}.

According to the estimates obtained above, we have 
\begin{equation}\label{modul ux}
\sup_{x\in\mathbb R}\left|\frac{\partial u_{i-1}}{\partial x}\right|\leq\sqrt2\left(\int\limits_{-\infty}^{+\infty}\left(\frac{\partial u_{i-1}}{\partial x}\right)^2 dx \int\limits_{-\infty}^{+\infty}\left(\frac{\partial^2 u_{i-1}}{\partial x^2}\right)^2 dx\right)^\frac{1}{4}<C,
\end{equation}

and then

$$\bigg|-2\gamma\int\limits_{-\infty}^{+\infty} u_i u_{i-1} \frac{\partial u_{i-1}}{\partial x} x^{2m} dx\bigg|\leq \sup_{x\in \mathbb R}\left|\frac{\partial u_{i-1}}{\partial x}\right| \int\limits_{-\infty}^{+\infty}|u_i||u_{i-1}|x^{2m} dx$$

\begin{equation}\leq C \int\limits_{-\infty}^{+\infty}|u_i||u_{i-1}|x^{2m} dx\leq \frac{C}{2} \int\limits_{-\infty}^{+\infty} u^2_i x^{2m} dx+\frac{C}{2} \int\limits_{-\infty}^{+\infty} u^2_{i-1} x^{2m} dx.\end{equation}

For the second term on the right-hand side of the equality \eqref{ESTx^2m}, we have
\begin{equation}
\bigg|2\int\limits_{-\infty}^{+\infty} f(x,t) x^{2m} u_i dx \bigg|\leq \int\limits_{-\infty}^{+\infty} x^{4m} f^2(x,t) dx+\int\limits_{-\infty}^{+\infty} u_i^{2} dx\leq C_1.\end{equation}

Now we estimate integrals on the left-hand side of the equality \eqref{ESTx^2m}. According to Lemma \ref{LemmAli}, we have

\begin{equation}
2\int\limits_{-\infty}^{+\infty} x^{2m} u_i \partial^\beta_{0t} u_i dx\geq \partial^\beta_{0t} \int\limits_{-\infty}^{+\infty} x^{2m} u^2_i dx.\end{equation}

We simplify the other terms on the left-hand side by successively applying integration by parts.
$$2\int\limits_{-\infty}^{+\infty} x^{2m} u_i u_{ixxxx}dx=2 m(2m-1)(2m-2)(2m-3)\int\limits_{-\infty}^{+\infty} x^{2m-4} u^2_i dx $$

\begin{equation}
-8 m (2m-1)\int\limits_{-\infty}^{+\infty} x^{2m-2}u^2_{ix}dx+2\int\limits_{-\infty}^{+\infty} x^{2m} u^2_{ixx}dx.\end{equation}

Similarly, 
\begin{equation}
\int\limits_{-\infty}^{+\infty} x^{2m} u_i u_{ixxx} dx=M_1\int\limits_{-\infty}^{+\infty} x^{2m-3} u^2_i dx
+M_2\int\limits_{-\infty}^{+\infty} x^{2m-1} u^2_{ix} dx,\end{equation}

\begin{equation}
\int\limits_{-\infty}^{+\infty} x^{2m} u_i u_{ixx} dx=-\int\limits_{-\infty}^{+\infty} x^{2m} u^2_{ix} dx
+M_3\int\limits_{-\infty}^{+\infty} x^{2m-2} u^2_{i} dx,\end{equation}

\begin{equation}
2\int\limits_{-\infty}^{+\infty} x^{2m} u_i u_{ix} dx=-2m\int\limits_{-\infty}^{+\infty} x^{2m-1} u^2_i dx.\end{equation}

Now we need some inequalities for further estimations. For $v\in S(\mathbb R)$ and $k\in \mathbb N$ we have 

\begin{equation}
\int\limits_{-\infty}^{+\infty} x^{k} v^2 dx\leq \int\limits_{|x|\geq 1} x^{k} v^2 dx+\int\limits_{-1}^{1} v^2 dx\leq \int\limits_{-\infty}^{+\infty} |x|^{k+1}v^2 dx+ \int\limits_{-\infty}^{+\infty}v^2 dx, \end{equation}
and
$$
\int\limits_{-\infty}^{+\infty} x^{k} v^2_x dx= -\int\limits_{-\infty}^{+\infty} x^{k} v v_{xx} dx+\frac{k(k-1)}{2}\int\limits_{-\infty}^{+\infty}x^{k-2}v^2 dx$$

$$\leq \sigma\int\limits_{-\infty}^{+\infty} |x|^{k} v_{xx}^2 dx +C(k,\sigma)\int\limits_{-\infty}^{+\infty}(|x|^{k}+1)v^2 dx.$$

Thus, taking into account the above estimates and choosing an appropriate value of $\sigma$, from \eqref{ESTx^2m} we obtain

\begin{equation}\label{BAHO}
\partial^\beta_{0t}\int\limits_{-\infty}^{+\infty} x^{2m} u^2_i
dx\leq M_5(m)\int\limits_{-\infty}^{+\infty} x^{2m} u^2_i dx+M_6
\int\limits_{-\infty}^{+\infty} x^{2m} u^2_{i-1} dx+M_7(m).\end{equation}
Here $M_6$ does not depend on $m$.

Applying the Gronwall--Bellman inequality from Lemma \ref{GRONWALL} to \eqref{BAHO}, we obtain 
$$\int\limits_{-\infty}^{+\infty} x^{2m} u^2_i dx\leq E_\beta(M_5T^\beta) \int\limits_{-\infty}^{+\infty} x^{2m}\varphi^2(x)dx $$

$$+
\frac{M_7T^\beta}{\beta} E_{\beta,\beta}(M_5T^\beta) + M_6 E_{\beta,\beta}(M_5T^\beta) I_{0t}^\beta \int\limits_{-\infty}^{+\infty} x^{2m} u^2_{i-1} dx=M_8+M_9I_{0t}^\beta\int\limits_{-\infty}^{+\infty} x^{2m} u^2_{i-1} dx.$$

Now, applying Lemma \ref{yangi lemma1} and taking into account $u_0(x,t)=\varphi(x)$, we obtain
$$\int\limits_{-\infty}^{+\infty} x^{2m} u^2_i dx\leq M_8 E_\beta(M_9T^\beta)+M_9^iI_{0t}^{i\beta}\int\limits_{-\infty}^{+\infty} x^{2m} \varphi^2(x) dx$$

$$=M_8 E_\beta(M_9T^\beta)+\frac{M_9^i T^{i\beta}}{\Gamma(i\beta+1)}\int\limits_{-\infty}^{+\infty} x^{2m} \varphi^2(x) dx \leq const,$$
as $$\lim_{i\to +\infty} \frac{M_9^i T^{i\beta}}{\Gamma(i\beta+1)}=0.$$
\end{proof}

A priori estimates with respect to other seminorms \eqref{yarim norma} can be obtained using equation \eqref{1-tenglama} and the inequality from Lemma \ref{Lemm_yakupov}.

\subsection{Convergence of $\{u_i\}$.}

Now, to complete the proof of Theorem \ref{nonlinear}, we should show that $\{u_i\}$ is a Cauchy sequence for each of the above seminorms \eqref{yarim norma}. 
 
Denote the difference $\omega_i=u_i-u_{i-1}$, $i\geq 1.$

Then for $i\geq 2$ we have the following equation for $\omega_i$.
\begin{equation}\label{difference}
    \partial _{0t}^{\beta }{{\omega }_{i}}+L\omega_{i}
=-\gamma{\omega }_{i-1}{{u}_{i-1,x}}-\gamma{{u}_{i-2}}{{\omega }_{i-1,x}}, \ \ \omega_i(x,0)=0.
\end{equation}

Now we multiply the  equation \eqref{difference} by $\omega_i$ and integrate over $\mathbb{R}$.

$$\int\limits_{-\infty}^{+\infty}{\partial _{0t}^{\beta }{{\omega }_{i}}{{\omega }_{i}}dx
+\int\limits_{-\infty}^{+\infty}L\omega_{i}}{{\omega }_{i}}dx
=-\gamma\int\limits_{-\infty}^{+\infty}{{{\omega }_{i-1}}\frac{\partial {{u}_{i-1}}}{\partial x}{{\omega }_{i}}dx}
-\gamma\int\limits_{-\infty}^{+\infty}{{u}_{i-2}}\left( \frac{\partial {{\omega }_{i-1}}}{\partial x} \right){{\omega }_{i}}dx=I_1+I_2.$$

We estimate the right-hand side. To do this, we use integration by parts and apply the Cauchy–Schwarz inequality.

In the first term, using the estimate \eqref{modul ux}, we get

\begin{equation}|{{I}_{1}}|=\left| \gamma\int\limits_{-\infty}^{+\infty}{{{\omega }_{i-1}}{{u}_{i-1,x}}{{\omega }_{i}}dx} \right|\le |\gamma|\underset{x\in \mathbb R}{\mathop{\sup }}\,\left| {{u}_{i-1,x}} \right|\int\limits_{-\infty}^{+\infty}{\left| {{\omega }_{i-1}} \right|\cdot \left| {{\omega }_{i}} \right|dx\le }
 M\left[ {{\left\| {{\omega }_{i}} \right\|}^{2}}+\left\| \omega _{i-1} \right\|^2 \right].\end{equation} 

Now we split the second integral into two parts

$${{I}_{2}}=-\gamma\int\limits_{-\infty}^{+\infty}{{{u}_{i-2}}{{\omega }_{i-1,x}}{{\omega }_{i}}dx=}\gamma\int\limits_{-\infty}^{+\infty}{{u}_{i-2}}{{\omega }_{ix}}{{\omega }_{i-1}}dx+\gamma\int\limits_{-\infty}^{+\infty}{{{u}_{i-2,x}}{{\omega }_{i}}{{\omega }_{i-1}}dx}=I_{2,1}+I_{2,2}.$$

Using the boundedness of $u_{i-2}$ and $u_{i-2,x}$, we can estimate these integrals as follows:

$$|I_{2,1}|\leq const \int\limits_{-\infty}^{+\infty} |\omega_{ix}||\omega_{i-1}|dx \leq \sigma\int\limits_{-\infty}^{+\infty} |\omega_{ix}|^2 dx+C_1(\sigma)\int\limits_{-\infty}^{+\infty}|\omega_{i-1}^2|dx$$
$$\leq -\sigma\int\limits_{-\infty}^{+\infty} \omega_{i}\omega_{ixx} dx+C_1(\sigma)\int\limits_{-\infty}^{+\infty}|\omega_{i-1}^2|dx $$
$$\leq \sigma\int\limits_{-\infty}^{+\infty} \omega_{ixx}^2 dx+C_1(\sigma)\int\limits_{-\infty}^{+\infty}\omega_{i-1}^2dx+4\sigma\int\limits_{-\infty}^{+\infty}\omega_{i}^2dx,$$

$$|I_{2,2}|\leq const \left(\int\limits_{-\infty}^{+\infty} (|\omega_{i}|^2 +|\omega_{i-1}^2|)dx\right).$$

Choosing  $\sigma>0$ sufficiently small, as in the preceding estimates, we get 

$$\partial _{0t}^{\beta }{{\left\| {{\omega }_{i}} \right\|}^{2}}\le \widetilde{C}\left[ {{\left\| {{\omega }_{i}} \right\|}^{2}}+{{\left\| {{\omega }_{i-1}} \right\|}^{2}} \right].$$

From the last inequality, using Gronwall's inequality from Lemma \ref{GRONWALL}, we get

\begin{equation}
{{\left\| {{\omega }_{i}}(\cdot,t) \right\|}^{2}} \leq {{\left\| {{\omega }_{i}}(\cdot,0) \right\|}^{2}}E_\beta(\widetilde{C}t^\beta)+\widetilde{C} \cdot\Gamma(\beta)E_{\beta,\beta}(\widetilde{C}T^\beta)I^\beta _{0t}{{\left\| {{\omega }_{i-1}}(\cdot,t) \right\|}^{2}} = KI^\beta _{0t}{{\left\| {{\omega }_{i-1}}(x,t) \right\|}^{2}}, 
\end{equation}
where $K=\widetilde{C} \cdot\Gamma(\beta)E_{\beta,\beta}(\widetilde{C}T^\beta)$. 
Here we used the fact that $\omega_i(x,0)=u_i(x,0)-u_{i-1} (x,0)=0$.

Therefore, by Lemma \ref{yangi lemma1}, the following inequality is obtained:

\begin{equation}
{{\left\| {{\omega }_{i}}(\cdot,t) \right\|}^{2}} \leq K^{i-1}I_{0t}^{(i-1)\beta}{{\left\| {{\omega }_{1}}(\cdot,t) \right\|}^{2}}.\end{equation}

From the above estimates, since
$\|\omega _1 (\cdot,t)\|=\|u _1 (\cdot,t)-u_0(\cdot,t)\|\leq const$
we have
\begin{equation}
{{\left\| {{\omega }_{i}}(\cdot,t) \right\|}^{2}} \leq K^{i-1}C\frac{T^{(i-1)\beta}}{\Gamma\left((i-1)\beta+1\right)}.\end{equation}

Now we prove that the sequence $\{u_i\}$ is a Cauchy (fundamental) sequence with respect to the norm $\|\cdot\|$.
We have

$$\|u_{i+p} (\cdot,t)-u_i (\cdot,t)\|=\left\|\sum_{k={i+1}}^{i+p}\omega_k (\cdot,t)\right\|\leq\sum_{k={i+1}}^{i+p}\|\omega_k (\cdot,t)\|\leq C\sum_{k=i+1}^{i+p}\frac{(KT^\beta)^\frac{{k-1}}{2}}{\sqrt {\Gamma((k-1)\beta+1)}}.$$

It is well known that the series 
$\sum_{k=1}^{\infty}\frac{(KT^\beta)^\frac{{k-1}}{2}}{\sqrt {\Gamma((k-1)\beta+1)}}$
is a convergent series. Therefore, by the Cauchy criterion,

$$\forall\varepsilon>0,\exists N: \forall i\geq N, \ \ \sum_{k=i+1}^{\infty}\|\omega_k(\cdot,t)\|<\varepsilon.$$
Consequently,
$$\max_{0\leq t\leq t_1}\left\|u_{i+p} (\cdot,t)-u_i (\cdot,t)\right\|<\varepsilon.$$
This means that the sequence $u_i (x,t)$ is a Cauchy (fundamental) sequence in the norm of $C([0,t_1], L^2(\mathbb R))$. 

By the standard interpolation inequality for derivatives
(see, e.g., \cite[Chapter~5]{AdamsFournier}),
$$ |||v|||_j\leq |||v|||_m^{j/m} ||v||^{1-j/m}, \qquad 0\le j\le m. $$

For every positive integer $j$ we can choose $N>j$ and according to Lemma \ref{high derivatives} we have $|||u_i|||_N\leq \sqrt{ M(N)},$ where the constant is independent of $i$ and $t$. 
Using this inequality, we have 
$$ |||u_{i+p} (x,t)-u_i (x,t)|||_j\leq |||u_{i+p} (x,t)-u_i (x,t)|||_N^{j/N}
||u_{i+p} (x,t)-u_i (x,t)||^{1-j/N}\leq $$ 

$$ \leq \left( 2\sqrt{M(N)}\right)^{j/N}||u_{i+p} (x,t)-u_i (x,t)||^{1-j/N}.$$
Thus, $\{u_i\}$ is a Cauchy sequence in the seminorms $\max_{t}|||\cdot |||_j$ for $j=1,2,3,...$ .

Now we show convergence in weighted norms. Using the Cauchy–Schwarz inequality and estimates from the Lemma \ref{estimate x_mu} we have
$$ \int_{\mathbb R}x^{2m}(u_{i+p} (x,t)-u_i (x,t))^2dx\leq \|u_{i+p} (x,t)-u_i (x,t)\| \left(\int_{\mathbb R}x^{4m}(u_{i+p} (x,t)-u_i (x,t))^2dx\right)^{1/2}\leq$$
$$\leq 2\sqrt{M_1(2m)} \|u_{i+p} (x,t)-u_i (x,t)\|.$$

From the above estimates, we conclude that the sequence $\{u_i\}$ converges in each of the seminorms \eqref{yarim norma} uniformly with respect to $t$.

Convergence of $\partial^{\beta}_{0t}u_i$ in the same seminorms uniformly with respect to $t$ follows directly from the equality 
$$\partial _{0t}^{\beta }{{u}_{i}} =-a^2u_{_ixxxx}-bu_{ixxx}-cu_{ixx}-du_{ix}-ku_i+f(x,t)-\gamma{{u}_{i-1}}{{u}_{i-1,x}}.$$

Since the right-hand side converges in $C([0,t_1];S(\mathbb R))$, we have
$\partial_{0t}^{\beta}u_i\to q:=f-Lu-\gamma uu_x \quad\text{in } C([0,t_1];S(\mathbb R)). $

On the other hand, $u_i-\varphi=I_{0t}^{\beta}\partial_{0t}^{\beta}u_i.$ Passing to the limit gives
$u-\varphi=I_{0t}^{\beta}q,$ and hence, by the definition of the regularized Caputo derivative,
$\partial_{0t}^{\beta}u=q.$

This completes the proof of Theorem~\ref{nonlinear}.

\section{GLOBAL SOLVABILITY} 

In this section, we discuss the problem of the continuation of the solution to the interval $[0,T]$ for arbitrary $T>0$. Here, we suppose that for some $0<t_0<T$ the solution of the problem is constructed for $0\leq t\leq t_0$. 

So, we consider the continued problem
\begin{equation}\label{continued-equation}
\partial_{t_0t}^\beta u+Lu=-\gamma u u_x +f_{t_0}(x,t), \  t_0<t\leq t_0+\delta,\ x\in \mathbb R,\end{equation}
\begin{equation}\label{continued-equation-initial} u|_{t={t_0}+0}=\varphi_{t_0}(x), \ \  x\in \mathbb R,\end{equation}
where $f_{t_0}=f-\mathcal H_{{t_0}}u$, the initial data are taken from the old solution $\varphi_{t_0}=u|_{t={t_0}-0}$,
and the history term is defined by
\begin{equation}\label{history-term-generalized}
\mathcal H_{t_0} u(x,t):= \partial_{0t}^\beta u-\partial_{{t_0} t}^\beta u
=\frac{1}{\Gamma(1-\beta)}\left[
\frac{u({t_0})-u(0)}{(t-{t_0})^\beta}
-\beta\int_0^{t_0}\frac{u(s)-u(0)}{(t-s)^{\beta+1}}\,ds
\right], t>{t_0}.
\end{equation}

\textbf{Remark.} {\it From the above equations one can see that continuation in the case of a time-fractional equation is different from that in the case of an integer-order equation. More precisely, in this case one should take into account the "memory effect" of the fractional derivative. For continuation of the solution, we need to obtain an estimate for $\sup_t \|\mathcal H_{t_0} u(\cdot,t)\|_2$ that is uniform in ${t_0}$.}


We first prove a uniform estimate for this history term.

\begin{lem}\label{memory-lemma}
Let $X$ be a Banach space, let $0<t_0<T$, let $u\in C([0,t_0];X)$, and assume that
$$
g(t)=\partial_{0t}^{\beta}u(t)\in C([0,{t_0}];X).
$$
Then for $t>{t_0}$,
\begin{equation}\label{memory-representation}
\mathcal H_{t_0} u(t)
=\frac{\sin(\pi\beta)}{\pi}(t-{t_0})^{1-\beta}
\int_0^{t_0}\frac{({t_0}-s)^{\beta-1}}{t-s}\,g(s)\,ds.
\end{equation}
Consequently,
\begin{equation}\label{memory-estimate}
\|\mathcal H_{t_0} u(t)\|_X
\leq \max_{0\leq s\leq {t_0}}\|g(s)\|_X,
\end{equation}
and the constant in this estimate is independent of ${t_0}$ and $t$. Moreover, if we set $\mathcal H_{t_0} u({t_0})=g({t_0}-0)$, then $\mathcal H_{t_0} u\in C([{t_0},T];X)$.
\end{lem}

\begin{proof}
We put $g(t)=\partial_{0t}^{\beta}u(t).$ By the inversion formula
\begin{equation}\label{u-from-g}
u(s)-u(0)=I_{0s}^{\beta}g(s)
=\frac{1}{\Gamma(\beta)}\int\limits_0^s(s-r)^{\beta-1}g(r)\,dr.
\end{equation}
Substituting \eqref{u-from-g} into \eqref{history-term-generalized}, changing the order of integration, and simplifying, we obtain
$$
\mathcal H_{t_0} u(t)
=\frac{1}{\Gamma(\beta)\Gamma(1-\beta)}
\int\limits_0^{t_0}\left[
(t-{t_0})^{-\beta}({t_0}-r)^{\beta-1}
-\beta\int\limits_r^{t_0}(t-s)^{-\beta-1}(s-r)^{\beta-1}\,ds
\right]g(r)\,dr=$$

$$=\frac{\sin\beta\pi}{\pi}
\int\limits_0^{t_0}\left[
(t-{t_0})^{-\beta}({t_0}-r)^{\beta-1}
-\frac{({t_0}-r)^\beta}{(t-r)(t-{t_0})^\beta}\right]g(r)\,dr=$$

$$=\frac{\sin\beta\pi}{\pi}(t-{t_0})^{1-\beta}
\int\limits_0^{t_0}\frac{({t_0}-r)^{\beta-1}}{t-r}g(r)\,dr.$$
We obtained \eqref{memory-representation}.

The kernel in \eqref{memory-representation} is nonnegative. Furthermore, using the substitution
$s={t_0}-r(t-{t_0})$, we have
$$
\frac{\sin(\pi\beta)}{\pi}(t-{t_0})^{1-\beta}
\int\limits_0^{t_0}\frac{({t_0}-s)^{\beta-1}}{t-s}\,ds
=\frac{\sin(\pi\beta)}{\pi}
\int\limits_0^{{t_0}/(t-{t_0})}\frac{r^{\beta-1}}{1+r}\,dr\leq \frac{\sin(\pi\beta)}{\pi}\int\limits_0^{+\infty}\frac{r^{\beta-1}}{1+r}\,dr=1.
$$
This proves \eqref{memory-estimate}. 

The continuity for $t>t_0$ follows directly from
\eqref{memory-representation}. It remains to prove the continuity at
$t=t_0$.

Put $h=t-t_0>0$, $A=\frac{t_0}{h}$. 
Then 
$$
\mathcal H_{t_0}u(t)
=
\frac{\sin(\pi\beta)}{\pi}
\int\limits_0^A
\frac{r^{\beta-1}}{1+r}
g(t_0-rh)\,dr.
$$
Since
$$
\frac{\sin(\pi\beta)}{\pi}
\int\limits_0^{+\infty}
\frac{r^{\beta-1}}{1+r}\,dr=1,
$$
we have
$$
\begin{aligned}
\mathcal H_{t_0}u(t)-g(t_0)
={}&
\frac{\sin(\pi\beta)}{\pi}
\int\limits_0^A
\frac{r^{\beta-1}}{1+r}
\bigl[g(t_0-rh)-g(t_0)\bigr]\,dr
\\
&-
\frac{\sin(\pi\beta)}{\pi}
g(t_0)
\int\limits_A^{+\infty}
\frac{r^{\beta-1}}{1+r}\,dr.
\end{aligned}
$$

Let $\varepsilon>0$ and put
$$
M=1+\max_{0\leq s\leq t_0}\|g(s)\|_X.
$$
Since $
\int\limits_0^{+\infty}
\frac{r^{\beta-1}}{1+r}\,dr<+\infty,
$
we can choose $R>0$ so large that
$$
\frac{\sin(\pi\beta)}{\pi}
\int\limits_R^{+\infty}
\frac{r^{\beta-1}}{1+r}\,dr
<
\frac{\varepsilon}{6M}.
$$

Since $g\in C([0,t_0];X)$, there exists $\rho>0$ such that
$\|g(s)-g(t_0)\|_X<\frac{\varepsilon}{2}$ whenever $|s-t_0|<\rho$.

Now choose $h>0$ sufficiently small so that $
h<\min\left\{\frac{\rho}{R},\frac{t_0}{R}\right\}
$.
Then $A=t_0/h>R$, and for $0\leq r\leq R$ we have
$|t_0-rh-t_0|=rh<\rho.$
Now we estimate
$$
\begin{aligned}
\|\mathcal H_{t_0}u(t)-g(t_0)\|_X
\leq{}&
\frac{\sin(\pi\beta)}{\pi}
\int\limits_0^R
\frac{r^{\beta-1}}{1+r}
\|g(t_0-rh)-g(t_0)\|_X\,dr
\\
&+
\frac{\sin(\pi\beta)}{\pi}
\int\limits_R^A
\frac{r^{\beta-1}}{1+r}
\|g(t_0-rh)-g(t_0)\|_X\,dr
\\
&+
\frac{\sin(\pi\beta)}{\pi}
\|g(t_0)\|_X
\int\limits_A^{+\infty}
\frac{r^{\beta-1}}{1+r}\,dr.
\end{aligned}
$$
The first term is bounded by
$$
\frac{\varepsilon}{2}
\frac{\sin(\pi\beta)}{\pi}
\int\limits_0^R\frac{r^{\beta-1}}{1+r}\,dr
\leq
\frac{\varepsilon}{2}.
$$
Moreover, since
$$
\|g(t_0-rh)-g(t_0)\|_X\leq 2M
$$
and $A>R$, the sum of the last two terms is bounded by
$$
3M\frac{\sin(\pi\beta)}{\pi}
\int\limits_R^{+\infty}
\frac{r^{\beta-1}}{1+r}\,dr
<
\frac{\varepsilon}{2}.
$$
Hence
$$
\|\mathcal H_{t_0}u(t)-g(t_0)\|_X<\varepsilon
$$
for all sufficiently small $t-t_0>0$. Therefore,
$$
\lim_{t\to t_0+0}
\|\mathcal H_{t_0}u(t)-g(t_0)\|_X=0.
$$

Thus, defining $\mathcal H_{t_0}u(t_0)=g(t_0),$ we obtain $\mathcal H_{t_0}u\in C([t_0,T];X).$
\end{proof}


We now derive the estimates needed for continuation. Below we obtain estimates independent of ${t_0}$.

Multiplying equation~\eqref{1-tenglama} by $2u$, integrating over $\mathbb R$, and using Lemma~\ref{LemmAli} together with
$$
\int\limits_{-\infty}^{+\infty}u^2u_x\,dx=0,
$$
we obtain, after estimating the lower-order linear terms,
$$
\partial_{0t}^{\beta}\|u\|^2
+2a^2\|u_{xx}\|^2
\leq C_0\|u\|^2+\|f\|^2.
$$
So, Lemma~\ref{GRONWALL} gives
\begin{equation}\label{global-P0}
\max_{0\leq t\leq {t_0}}\|u(\cdot,t)\|\leq P_0(T),
\end{equation}
where $P_0(T)$ depends only on $T$, the coefficients, and the prescribed data.

Next, differentiating the equation once with respect to $x$ and multiplying by $2u_x$, we use

$$\left|\int\limits_{-\infty}^{+\infty}u_x (uu_x)_x dx\right|=\left|\int\limits_{-\infty}^{+\infty}u_{xx} u u_x dx\right|\leq \sup_{x \in \mathbb R} |u_x|\cdot\|u\|\cdot\|u_{xx}\|\le$$
$$\leq P_0 \|u_x\|^{1/2}\|u_{xx}\|^{3/2} \leq P_0 \|u_x\|^{5/4}\|u_{xxx}\|^{3/4}\leq \sigma \|u_{xxx}\|^2 +C(\sigma)\|u_x\|^2.$$

Hence
$$
\partial_{0t}^{\beta}\|u_x\|^2
+(2a^2-|\gamma|\sigma)\|u_{xxx}\|^2
\leq C_1(T)\|u_x\|^2+\|f_x\|^2,
$$
and by choosing $\sigma=a^2/|\gamma|$, Lemma~\ref{GRONWALL} yields
\begin{equation}\label{global-P1}
\max_{0\leq t\leq {t_0}}\|u_x(\cdot,t)\|\leq P_1(T).
\end{equation}

Now, take the second-order $x$-derivative and multiply both sides of the equation \eqref{1-tenglama} by $2u_{xx}$.
In this case, the integral arising from the nonlinear term can be estimated as follows
$$\left|\int\limits_{-\infty}^{+\infty}u_{xx} (uu_x)_{xx} dx\right|=\left|\int\limits_{-\infty}^{+\infty}u_{xxxx} u_{x} u dx\right|$$
$$\leq \sup_{x \in \mathbb R}|u| \cdot\|u_x\|\cdot\|u_{xxxx}\|\leq \sqrt{2}P_0^{1/2}P_1^{3/2}\|u_{xxxx}\|\leq C(\sigma_1)+\sigma_1\|u_{xxxx}\|^2.$$
Again, choosing an appropriate value of $\sigma_1$, we get 
$$
\partial_{0t}^{\beta}\|u_{xx}\|^2
+a^2\|u_{xxxx}\|^2
\leq C_2(T)\bigl(1+\|u_{xx}\|^2\bigr)+\|f_{xx}\|^2,
$$
and consequently
\begin{equation}\label{global-P2}
\max_{0\leq t\leq {t_0}}\|u_{xx}(\cdot,t)\|\leq P_2(T).
\end{equation}

For the higher derivatives we repeat the argument used in Lemma~\ref{high derivatives}, now for the solution itself. For $n\geq3$,
$$
\int_{\mathbb R}\frac{\partial^n}{\partial x^n}(uu_x)\,\frac{\partial^n u}{\partial x^n}\,dx
=\frac12\int_{\mathbb R}\frac{\partial^{n-1}}{\partial x^{n-1}}(u^2)\,\frac{\partial^{n+2} u}{\partial x^{n+2}}\,dx.
$$
By Leibniz' formula, every derivative entering $\frac{\partial^{n-1}}{\partial x^{n-1}}(u^2)$ has order strictly smaller than $n$. Thus, using the same estimates as in the proof of Lemma~\ref{high derivatives}, by induction we obtain 
$$
\partial_{0t}^{\beta}\| u\|_n^2
+a^2\|u\|_{n+2}^2
\leq C_n(T)\bigl(1+\| u\|_n^2\bigr)
+\|f\|_n^2.
$$
Applying Lemma~\ref{GRONWALL}, we obtain
\begin{equation}\label{global-Pn}
\max_{0\leq t\leq {t_0}}\| u(\cdot,t)\|_n\leq P_n(T),
\end{equation}
where none of the constants $P_n(T)$ depends on ${t_0}$.

We next estimate the full generalized Caputo derivative in the $H^2(\mathbb R)$ norm
$$
\|\partial_{0t}^{\beta}u\|_2=\|f-Lu-\gamma uu_x\|_2\leq \|f\|_2+\|Lu\|_2+|\gamma| \|uu_x\|_2.
$$
Estimates \eqref{global-P0}--\eqref{global-Pn} imply
$$\|Lu(t)\|_2\leq C\sum_{j=0}^{6}P_j(T).$$

For the nonlinear term we use
$$(uu_x)_{xx}=3u_xu_{xx}+uu_{xxx}.$$ 
Hence, by the one-dimensional Sobolev inequality and the already obtained estimates, $ \|uu_x\|_2\leq C_N(T).$

From the above estimates, it follows that
\begin{equation}\label{Caputo-uniform}
\max_{0\leq t\leq {t_0}}
\|\partial_{0t}^{\beta}u(t)\|_2\leq G(T),
\end{equation}
where $G(T)$ is independent of ${t_0}$.

We can now continue the solution. 
Because $u\in SC_\beta({t_0})$, the function
$g=\partial_{0t}^{\beta}u$ belongs to $C([0,{t_0}];S(\mathbb R))$. Formula~\eqref{memory-representation}, together with the triangle inequality for each Schwartz seminorm, shows that
$\mathcal H_{t_0} u\in C([{t_0},T];S(\mathbb R))$; hence $f_{t_0}$ is an admissible source term for the local theorem. More importantly, applying \eqref{memory-estimate} in the norm $\|\cdot\|_2$ and using \eqref{Caputo-uniform}, we obtain
\begin{equation}\label{history-uniform}
\|\mathcal H_{t_0} u(t)\|_2\leq G(T),\qquad {t_0}\leq t\leq T.
\end{equation}
Therefore
\begin{equation}\label{source-uniform}
\max_{{t_0}\leq t\leq T}\|f_{t_0}(t)\|_2
\leq \max_{0\leq t\leq T}\|f(t)\|_2+G(T)=:F_*(T),
\end{equation}
where $F_*(T)$ is independent of ${t_0}$. Also, from \eqref{global-P0} and \eqref{global-P2},
\begin{equation}\label{initial-uniform}
\|\varphi_{t_0}\|_2^2
=\|u({t_0})\|^2+\|u_{xx}({t_0})\|^2
\leq P_0(T)^2+P_2(T)^2=:U_*(T).
\end{equation}

Set $s=t-{t_0},$ $v(x,s)=u(x,{t_0}+s)$, $\widetilde f_{t_0}(x,s)=f_{t_0}(x,{t_0}+s)$.

Then problem \eqref{continued-equation}--\eqref{continued-equation-initial} takes the form
$$
\partial_{0s}^{\beta}v+Lv
=
-\gamma vv_x+\widetilde f_{t_0}(x,s),
\qquad 0<s\leq T-{t_0},
$$
$$
v(x,0)=\varphi_{t_0}(x).
$$
Thus, Theorem~\ref{nonlinear} is applicable to the shifted problem. The quantity
determining the local existence length satisfies
$$
\begin{aligned}
A_{t_0}
={}&
\frac{c_8(T-{t_0})^\beta}{\beta}
E_{\beta,\beta}\!\left(c_6(T-{t_0})^\beta\right)
\max_{{t_0}\leq t\leq T}\|f_{t_0}(t)\|_2^2 \\
&+
E_\beta\!\left(c_6(T-{t_0})^\beta\right)
\|\varphi_{t_0}\|_2^2
\\
\leq{}&
\frac{c_8T^\beta}{\beta}
E_{\beta,\beta}(c_6T^\beta)
\max_{{t_0}\leq t\leq T}\|f_{t_0}(t)\|_2^2
+
E_\beta(c_6T^\beta)\|\varphi_{t_0}\|_2^2
\\
\leq{}& \widetilde A(T).
\end{aligned}
$$

The number $\widetilde A(T)$ is independent of the continuation point ${t_0}$. Consequently, at every continuation step we may use the same positive length
\begin{equation}\label{uniform-continuation-step}
\delta
=\left(
\frac{\beta}{4c_7\widetilde A(T)E_{\beta,\beta}(c_6T^\beta)}
\right)^{1/\beta}>0.
\end{equation}

The local theorem therefore provides a solution on $[{t_0},{t_0}+\delta]$ (or up to $T$ if ${t_0}+\delta>T$). According to Lemma \ref{memory-lemma}, the glued solution belongs to $SC_\beta$ across the continuation point.

At every new continuation point, the modified source term $f_{t_0}=f-\mathcal H_{t_0} u$ incorporates the complete fractional memory accumulated from time $0$ up to that point, while estimates \eqref{source-uniform} and \eqref{initial-uniform} remain unchanged. Hence the same $\delta$ can be used at every step. A finite number of steps therefore covers the whole interval $[0,T]$.

We have proved the following global solvability result.

\begin{thm} \label{Global} (Global solvability).
    Let $T$ be an arbitrary positive number, ${\varphi (x)}\in S(\mathbb R)$ and ${f(t,x)}\in SC(T)$. Then
the Cauchy problem \eqref{1-tenglama}, \eqref{1-tenhglama IC} has a solution   $u\in SC_\beta(T)$.
\end{thm}

\section{STABILITY OF THE SOLUTION}

\begin{thm}\label{Stability}
Let $T>0$, $\varphi,\psi\in S(\mathbb R)$ and
$f,h\in C([0,T];S(\mathbb R))$.
Let $u$ and $v$ be the corresponding solutions of the Cauchy problems
with initial data $\varphi$ and $\psi$, and source terms $f$ and $h$,
respectively. Then
$$
\max_{0\leq t\leq T}\|u(\cdot,t)-v(\cdot,t)\|^2
\leq
C_2\|\varphi-\psi\|^2
+
C_3\max_{0\leq t\leq T}
\|f(\cdot,t)-h(\cdot,t)\|^2,
$$
where $C_2$ and $C_3$ are positive constants depending on $T$,
the coefficients of the equation, and 
$$
R=
\|\varphi\|_2+\|\psi\|_2
+\max_{0\leq t\leq T}\|f(\cdot,t)\|_2
+\max_{0\leq t\leq T}\|h(\cdot,t)\|_2 .
$$
\end{thm}

\begin{proof}
We denote $w=u-v$ and obtain the following equation for $w$
\begin{equation}
\partial^\beta_{0t}w+Lw+\gamma (wu_x+vw_x)=f-h,\quad x\in \mathbb R, \ 0<t\leq T
\end{equation}
with initial condition $$w(x,0)=\varphi(x)-\psi(x), \quad x\in \mathbb R.$$
We multiply both sides of this equation by $2w$ and integrate over $\mathbb R$ with respect to $x$. Using similar steps as in the proof of Lemma \ref{norma2} and estimates from Section 4, we get 



$$
\partial_{0t}^{\beta}\|w(\cdot,t)\|^2 \leq C\|w(\cdot,t)\|^2 +
\|f(\cdot,t)-h(\cdot,t)\|^2, $$
where the constant $C$ depends on $T$, $R$, and the coefficients of the equation. 

By Lemma~\ref{GRONWALL},
$$\|w(\cdot,t)\|^2 \leq \|\varphi-\psi\|^2E_\beta(Ct^\beta)
+ \Gamma(\beta)E_{\beta,\beta}(Ct^\beta) I_{0t}^{\beta}
\|f(\cdot,t)-h(\cdot,t)\|^2. $$  

Since
$$
I_{0t}^{\beta}
\|f(\cdot,t)-h(\cdot,t)\|^2
\leq
\frac{T^\beta}{\Gamma(\beta+1)}
\max_{0\leq t\leq T}
\|f(\cdot,t)-h(\cdot,t)\|^2,
$$
we obtain
$$
\max_{0\leq t\leq T}\|u(\cdot,t)-v(\cdot,t)\|^2
\leq
C_2\|\varphi-\psi\|^2
+
C_3
\max_{0\leq t\leq T}
\|f(\cdot,t)-h(\cdot,t)\|^2,
$$
where
$C_2=E_\beta(CT^\beta)$, $C_3=\frac{T^\beta}{\beta}E_{\beta,\beta}(CT^\beta)$.
\end{proof}

\textbf{Remark.} {\it Uniqueness of the solution follows from the stability estimate.}

\section{CONCLUSION}

In this paper, we studied the Cauchy problem for a generalized time-fractional Kuramoto--Sivashinsky equation on $\mathbb{R}$ with the regularized Caputo derivative in the Schwartz space of rapidly decreasing functions. For the corresponding linear problem, the Fourier transform was used to derive an explicit representation of the solution in terms of Mittag--Leffler functions and to establish preservation of the Schwartz class.

For the nonlinear problem, we applied the method of successive approximations and obtained a priori estimates in the seminorms defining the Schwartz topology. These estimates, together with interpolation and weighted estimates, allowed us to prove convergence of the approximations in the full Schwartz topology and hence local solvability of the nonlinear problem.

The main contribution of the paper concerns the global continuation of the local solution in the presence of fractional memory. Since the regularized Caputo derivative is nonlocal in time, continuation from an intermediate time cannot be performed by simply restarting the problem with new initial data. To overcome this difficulty, we decomposed the fractional derivative into a shifted regularized Caputo derivative and a history term containing the complete memory from the initial time. An explicit representation of this history term was obtained, together with an estimate that is uniform with respect to the continuation point. Combined with the spatial a priori estimates, this yields a continuation step whose length can be chosen independently of the continuation point. Consequently, the local solution can be continued to any prescribed finite time interval.

Finally, an $L^2(\mathbb{R})$ continuous-dependence estimate was established. In particular, this estimate implies uniqueness of the solution. Thus, the results provide a solvability framework for the generalized time-fractional Kuramoto--Sivashinsky equation in the Schwartz space, combining local solvability in the full Schwartz topology, global continuation with preservation of the fractional memory, and stability of solutions.

\section*{Statements and Declarations}

\textbf{Funding.}
The authors declare that no funds, grants, or other support were
received during the preparation of this manuscript.

\textbf{Competing Interests.}
The authors have no relevant financial or non-financial interests
to disclose.

\textbf{Author Contributions.}
All authors contributed to the conception of the study, mathematical
analysis, and preparation of the manuscript. All authors read and
approved the final manuscript.

\textbf{Data Availability.}
Data sharing is not applicable to this article, as no datasets were
generated or analyzed during the current study.

\label{lastpage}


%
\end{document}